\documentclass[12pt]{article}
\usepackage{geometry}                
\usepackage{graphicx}
\usepackage{amssymb, enumerate}
\usepackage{amsmath}
\usepackage{amssymb}
\usepackage{amsthm}
\usepackage{float}
\usepackage{kantlipsum, hyperref}
\newtheorem{theorem}{Theorem}[section]
\newtheorem{lemma}[theorem]{Lemma}
\newtheorem{corollary}[theorem]{Corollary}

\newtheorem{remark}[theorem]{Remark}

\newtheorem{conjecture}[theorem]{Conjecture}

\newtheorem{definition}[theorem]{Definition}

\newcommand{\In}{\textnormal{In}}
\newcommand{\Out}{\textnormal{Out}}
\newcommand{\dist}{\textnormal{ood}}
\newcommand{\disti}{\textnormal{oid}}
\newcommand{\vin}{V_\text{in}}
\newcommand{\vout}{V_\text{out}}

\title{Antidirected subgraphs of oriented graphs}
\author{Maya Stein\thanks{MS is affiliated to the Center for Mathematical Modeling and the Department of Mathematical Engineering of the University of Chile, and acknowledges support by ANID Fondecyt Regular Grant 1221905, by FAPESP-ANID Investigaci\'on Conjunta grant 2019/13364-7, by RandNET (RISE project H2020-EU.1.3.3) and by ANID PIA CMM FB210005.} \ and Camila Zárate-Guerén\thanks{CZG is affiliated to the University of Birmingham and acknowledges support by the University of Birmingham, by RandNET (RISE project H2020-EU.1.3.3) and by ANID PIA CMM FB210005.}}

\date{}                                           

\begin{document}
\maketitle

\begin{abstract}
We show that 
for every $\eta>0$  every sufficiently large 
$n$-vertex oriented graph $D$ of minimum semi\-degree exceeding $(1+\eta)\frac k2$ contains  every  balanced  antidirected tree  with $k$ edges and bounded maximum degree, if $k\ge\eta n$. In particular, this asymptotically  confirms  a conjecture of the first author for long antidirected paths and dense digraphs.
\\
Further, we show that in the same setting, $D$ contains  every $k$-edge anti\-directed subdivision of a sufficiently small complete graph, if the paths of the subdivision that have length $1$ or $2$ span a forest. As a special case, we can find all antidirected cycles of length at most  $k$.
\\
Finally, we address a  conjecture of Addario-Berry, Havet, Linhares Sales, Reed and Thomass\'e for antidirected trees in digraphs. We show that this conjecture is asymptotically true
in $n$-vertex oriented graphs for all balanced antidirected trees of bounded maximum degree and of size linear in $n$.
\end{abstract}

\section{Introduction}

A typical question in extremal graph theory is whether  information  on the degree sequence of a graph $G$ can be used to find some specific subgraph in $G$. A famous example  is Dirac's theorem~\cite{dirac}, which states that any $n$-vertex graph $G$  of minimum degree $\delta (G)\ge \frac{n}{2}$ contains a spanning cycle.  
A  minimum degree of at least $\frac{n-1}{2}$  guarantees a spanning path. Such a result also holds for shorter paths: Already
Dirac, as well as Erdős and Gallai, observed that
if a graph~$G$ has a connected component with at least $k+1$ vertices and $\delta(G)\geq \frac k2$,  then $G$ contains a $k$-edge path (see~\cite{erdosgallai}). Note that we need to 
require that $G$ has a sufficiently large component as otherwise 
$G$ could be  the disjoint union of complete graphs of order $k$.

Considering the same question for oriented graphs $D$, we replace the  minimum degree with the minimum semidegree $\delta^0(D)$, which is defined as the minimum over all the in- and all the out-degrees of the vertices. 
Note that any oriented graph of minimum semidegree $\delta^0(D) > \frac k2$ has an underlying connected graph on more than $k$ vertices (and therefore, no extra condition on large components will be necessary). 
Jackson \cite{bjackson} 
showed that every oriented graph~$D$ with $\delta^0(D) > \frac k2$ contains the directed path on $k$ edges, which is best possible.
The first author conjectured that this result  extends to any orientation of the path.

\begin{conjecture}$\!\!${\rm\bf\cite{survey}}
\label{conjp}
Every oriented graph $D$ with $\delta^0(D) > \frac k2$ contains every oriented path on $k$ edges.
\end{conjecture}

This conjecture is best possible for odd $k$, as a regular tournament on $k$ vertices does not contain any oriented path with $k$ edges. A different example works for  {\it antidirected paths}, i.e.~oriented paths whose edges alternate directions (more generally, an {\it anti\-directed graph} is an oriented graph having no $2$-edge directed path).
Note that  in an $\ell$-blow-up of the directed triangle (where each vertex is replaced with $\ell$ independent vertices), any   longest antidirected path can only cover $2\ell$ vertices. 

It is known~\cite{antipaths} that if $\delta^0(D)> \frac{3k-2}{4}$, then the oriented graph $D$ contains each antipath with $k$ edges. Further, if  we replace the bound with $\delta^0(D) \ge k$, 
Conjecture~\ref{conjp} becomes very easy (a greedy embedding strategy suffices). 
This strategy   also works if we wish to find a $k$-edge oriented tree instead of an $k$-edge oriented path. Moreover, in this case we cannot do any better: In the $(k-1)$-blow-up of a directed triangle, each vertex has semidegree $k-1$, but no antidirected star with $k$ edges is present. However,  the antidirected star is  very unbalanced. The situation might be different for  {\it balanced} antidirected trees $T$, i.e.~those that have as many vertices of in-degree $0$ as  of out-degree $0$. We suspect that  these trees appear whenever the host oriented graph $D$ has   minimum semidegree greater than $\frac k2$.

We show that this is asymptotically  true   if $D$ is dense and $T$ is large and has bounded  maximum degree, where the {\it maximum degree} $\Delta(T)$ of an oriented tree $T$ is defined as the maximum degree of the underlying undirected tree. 

\begin{theorem}
\label{teo:teo1}
For all $\eta\in(0,1)$ and $c\in \mathbb N$ there is $n_0$ such that for all $n\geq n_0$ and $k\geq \eta n$,  every oriented graph $D$ on $n$ vertices with $\delta^0 (D)>(1 + \eta)\frac k2$ contains  every balanced antidirected tree $T$ with $k$ edges and with $\Delta(T)\leq (\log(n))^c$.
\end{theorem}
 Instead of  Theorem~\ref{teo:teo1} we will prove a slightly stronger version, namely Theorem~\ref{teo:teo1'}. This is necessary for a later application of the result in the proof of Theorem~\ref{ES} below. Theorem~\ref{teo:teo1'}  allows us to specify a subset of $V(D)$ where the root of $T$ will be mapped, a feature that might be of independent interest.
 
Theorem~\ref{teo:teo1}  can  be interpreted as a version  for smaller trees  of a recent result by Kathapurkar and Montgomery~\cite{km} (and the preceding result in~\cite{mn}). In~\cite{km} it is shown that for each $\eta > 0$, there is some $c > 0$  such that every sufficiently large $n$-vertex directed graph with minimum semidegree at least $(\frac 12 + \eta)n$ contains a copy of every $n$-vertex oriented tree whose underlying maximum degree is at most $c\frac n{\log n}$. This generalises a well-known theorem of Koml\'os, S\'ark\"ozy and Szemer\'edi~\cite{kss} for graphs. The result from~\cite{km} cannot hold for smaller trees and accordingly smaller semidegree  without adding further conditions, because of the example of the antidirected star from above.

Turning back to our original question for oriented paths, note that Theorem~\ref{teo:teo1} immediately implies that Conjecture \ref{conjp} holds asymptotically for large antidirected paths in large graphs. (If the path is not balanced, we can extend it by one vertex, and apply Theorem~\ref{teo:teo1} with a sufficiently smaller $\eta$.)

\begin{corollary}
\label{coro1}
For all $\eta\in(0,1)$  there is $n_0$ such that for all $n\geq n_0$ and $k\geq \eta n$ every oriented graph $D$ on $n$ vertices with $\delta^0 (D)>(1 + \eta)\frac k2$ contains  every  antidirected path with $k$ edges.
\end{corollary}

How about asking for oriented cycles instead of oriented paths or trees? 
Keevash, Kühn and Osthus \cite{keeko} showed  that every sufficiently large $n$-vertex oriented graph $G$ of minimum semidegree $\delta^0(G)\geq \frac{3n}8$ contains  a  directed Hamilton cycle, and  Kelly~\cite{kelly} showed that asymptotically, the same bound  guarantees any orientation of a Hamilton cycle.
Kelly, Kühn and Osthus \cite{kko1} extended this to a pancyclicity result, proving  that every  oriented graph~$D$ of minimum semidegree $\delta^0(D)\geq \frac{3n}8+o(n)$ contains every orientation of every cycle (of any length between $3$ and $n$). These authors also show that for cycles of constant length  the bounds can be improved, but the improvement depends on the {\it cycle type}: the number of forward edges minus the number of backwards edges of the cycle. For some cycle types and lengths, a semidegree exceeding~$\frac n3$ is necessary:  for instance, if $3$ does not divide $k$, a blowup of the directed triangle has minimum semidegree~$\frac n3$ but does not contain a directed $C_k$. For cycle type $0$ the bound on the minimum semidegree can be much lower: For each $k\ge 3$ and $\eta>0$ every large enough $n$-vertex oriented graph of minimum semidegree at least~$\eta n$ contains all oriented cycles of length at most $k$ and cycle type $0$~\cite{kko1}.

As antidirected cycles have cycle type $0$, the  discussion from the previous paragraph implies  that we can find short antidirected cycles with a minimum semidegree of $\eta n$, and antidirected cycles of any length with a minimum semidegree exceeding $\delta^0(D)\geq \frac{3n}8+o(n)$. 
In light of Theorem~\ref{teo:teo1} it seems natural to suspect that an intermediate bound on the minimum semidegree could suffice for antidirected cycles of medium length. This is indeed possible, as our next result shows.

\begin{theorem}
\label{teo:teo2}
For all $\eta\in(0,1)$ there is $n_0$ such that for all $n\geq n_0$ and $k\geq \eta n$, every  oriented graph $D$ on $n$ vertices with $\delta^0 (D)>(1 + \eta)\frac k2$ contains any  antidirected cycle of length at most $k$.
\end{theorem}

We will in fact prove a more general result,  Theorem~\ref{teo:teo3} below.  Theorem~\ref{teo:teo3} focuses on embedding antidirected subdivisions of complete graphs. Before stating Theorem~\ref{teo:teo3}, let us review the history of (oriented) subdivisions in (oriented) graphs of high minimum (semi)degree. 

For undirected graphs, Mader~\cite{maderundir} proved that there is a function $g(h)$ such that every graph with  minimum degree at least $g(h)$ contains a subdivision of the complete graph $K_h$. Thomassen~\cite{tho} showed that a direct translation of this result to digraphs is not true: For every function $g(h)$ there is a digraph of minimum outdegree $g(h)$ that does not contain a subdivision of the complete digraph on $h\ge 3$ vertices. (A subdivision of a digraph $D$  substitutes each edge of $D$ with a {directed} path.) Mader~\cite{mader} modified Thomassen's construction  showing that the same is true if we replace the minimum outdegree with the   minimum semidegree. Crucially, the constructed digraphs have no even directed cycles (which appear in subdivisions of~$K_3$). Following these discoveries, Mader suggested to replace the subdivision of the complete digraph with 
the {\it transitive tournament}, i.e.~the tournament without directed cycles. He conjectured the following.

\begin{conjecture}[Mader \cite{mader}]\label{mader}
There is a function  $f(h)$ such that every digraph of minimum outdegree at least $f(h)$ contains a subdivision of the transitive tournament of order $h$.
\end{conjecture}

This conjecture is open even for $h=5$. Recently, there has been much activity on variants of  Conjecture~\ref{mader}, see for instance~\cite{achlmt, chln, gps, gss, gss2, jagger, kky, ste}.
Aboulker, Cohen, Havet, Lochet, Moura and Thomass\'e~\cite{achlmt} observed that Conjecture~\ref{mader} is equivalent to the following conjecture.

\begin{conjecture}$\!\!${\rm\bf\cite{achlmt}}\label{conj3}
There is a function  $f(h)$ such that every digraph of minimum semidegree at least $f(h)$ contains a subdivision of a transitive tournament of order~$h$.
\end{conjecture}

We will show a result  along the lines of Conjectures~\ref{mader} and~\ref{conj3} for oriented graphs. Namely, we will prove that any oriented graph of large minimum semidegree contains an antidirected orientation of a subdivision of a complete graph~$K_h$, where~$h$ cannot be too large. We can even choose the lengths of the antidirected paths in the subdivision, as long as the very short paths do not span a cycle. To make this more precise, we need some notation.

For $h,k\in \mathbb N$, consider a subdivision of $K_h$ where each edge $e\in E(K_h)$ is substituted by a path of  length  $g(e)$,  with $\sum_{e\in E(K_h)}g(e)=k$. Call any antidirected orientation of this graph a {\it $k$-edge antisubdivision of $K_h$}. 
If furthermore the edges of $K_h$ with $g(e)< 3$ induce a forest in $K_h$, we say the $k$-edge antisubdivision is {\it long}. Note that in particular, if all edges of $K_h$ are substituted with paths of length at least 3, the resulting antisubdivision is long.

\begin{theorem}
\label{teo:teo3}
For all $\eta\in(0,1)$ there are $n_0\in\mathbb N$ and $\gamma>0$ such that for each 
$n\ge n_0$, each $h\le \gamma \sqrt n$ and  each $k\ge \eta n$  the following holds.
Every  oriented graph $D$ on $n$ vertices with $\delta^0 (D)>(1 + \eta)\frac k2$   contains each long $k$-edge antisubdivision of~$K_h$.
\end{theorem}

As indicated above, Theorem~\ref{teo:teo3} quickly implies Theorem~\ref{teo:teo2}.\footnote{Indeed, it suffices to note that any antidirected cycle with more than $4$ edges can be interpreted as a long antisubdivision of $K_3$, and so,  Theorem~\ref{teo:teo2} follows directly from Theorem~\ref{teo:teo3} if the antidirected cycle we are looking for is larger than the minimum semidegree of $D$. If we are looking for a shorter antidirected cycle $C$ of length at least $6$, we can complete $C$ to a long $k$-edge antisubdivision of $K_4$, which we can find with the help of Theorem~\ref{teo:teo3}. Finally, an antidirected $C_4$  exists by the results of~\cite{kko}.}

Let us now turn back to oriented trees. Looking for parameters related to the appearance of certain oriented trees in a digraph, a natural alternative to  the mini\-mum semidegree is the edge density.
 In 1970, Graham~\cite{graham} confirmed a conjecture he attributes to Erd\H os: for every antidirected tree $T$ there is a constant $c_T$ such that every sufficiently large directed graph $D$ on $n$ vertices and with at least $c_Tn$ edges contains~$T$. A similar statement does not hold for other oriented trees: 
A bipartite graph on sets $A$ and $B$, with every edge oriented from $A$ to $B$, has $|A||B|$ edges, and only has antidirected subgraphs~\cite{ahlrt,burr}. 

In 1982, Burr~\cite{burr} gave an improvement of Graham's result: Every $n$-vertex digraph~$D$ with more than $4kn$ edges contains each antidirected tree $T$ on $k$ edges. He obtains this result by greedily embedding~$T$ into a suitable bipartite subgraph of~$D$. Burr states that the bound $4kn$ on the number of edges can `almost certainly be made rather smaller', and provides an example where $(k-1)n$ edges are not sufficient: The $k$-edge star with all edges directed outwards, and the complete bipartite graph $K_{k-2,k-2}$ with half of the edges  oriented in either direction in an appropriate way (one can also take the $(k-1)$-blow-up of the directed triangle).
In 2013, 
Addario-Berry, Havet, Linhares Sales, Reed and Thomassé~\cite{ahlrt} formulated a conjecture which states that $(k-1)n$ is the correct bound.

\begin{conjecture}$\!\!${\rm\bf\cite{ahlrt}}
\label{ESantitrees} Every $n$-vertex digraph $D$ with more than $(k-1)n$ edges contains each antidirected tree on $k$ edges.
\end{conjecture}

The authors of \cite{ahlrt} prove this conjecture for antidirected trees of diameter at most~3, and point out that 
it implies the well-known Erd\H os--S\'os conjecture for graphs.
If the digraph $D$ from  Conjecture~\ref{ESantitrees} is an {oriented}  graph, it is known~\cite{antipaths} that 
 about $\frac 32 (k-1)n$ edges are sufficient to find all antidirected paths on $k$ edges. 
 
We show that Conjecture~\ref{ESantitrees} is approximately true in oriented graphs  for all balanced antidirected trees of bounded maximum degree. 

\begin{theorem}
\label{ES}
For all $\eta\in(0,1)$ and $c\in\mathbb N$, there is $n_0\in\mathbb N$  such that for every $n\ge n_0$ and every $k\ge \eta  n$, every $n$-vertex oriented graph $D$  with more than $(1 + \eta)(k-1)n$ edges contains each balanced antidirected tree $T$ with $k$ edges and $\Delta(T)\leq (\log(n))^c$.
\end{theorem}

In particular, Conjecture~\ref{ESantitrees} holds asymptotically for long antidirected paths in oriented graphs. Moreover, Theorem~\ref{ES} implies the asymptotic bound $2(k-1)n$ for digraphs (as any digraph can be turned into an oriented graph by deleting at most half the edges).

The paper is organised as follows.  We start with a sketch of the proofs of Theorems \ref{teo:teo2}, \ref{teo:teo3} and \ref{ES} in Section~\ref{sketch}. We then go through some preliminaries in Section~\ref{cap2}. We give the proof of Theorem \ref{teo:teo2} in Section~\ref{proofteo}.
We prove Theorem \ref{teo:teo3} in Section~\ref{sec:subdi}, and Theorem~\ref{ES} in Section~\ref{edge-dens}.
In Section~\ref{conclusion}, we discuss some open questions.

\section{Proof sketches}
\label{sketch}
\paragraph{Connected Antimatchings.}
Fundamental to our proofs is
the concept of a \textit{connected antimatching}: this  is a set $M$ of disjoint edges in an oriented graph $D$ such that  every pair of edges in $M$  is connected by an antidirected walk in $D$ (see Definition \ref{def-antimatching}).    Section~\ref{sec:antiM} is dedicated to connected antimatchings. First, we will prove that any oriented graph of large minimum semidegree  contains a large connected antimatching $M$.  Second, we will show that we can choose such an  antimatching~$M$ in a way that the edges of $M$ lie at bounded distance $d$ from each other.

\paragraph{Sketch of the proof of Theorem \ref{teo:teo2}}
Given an oriented graph $D$ and an antidirected tree~$T$ fulfilling the conditions of the theorem, we apply the digraph regularity lemma  to $D$ to find a partition into a bounded number of clusters.  The reduced oriented graph~$R$ maintains the  condition on the minimum semidegree, and thus has a large connected antimatching whose edges lie at bounded distance from each other.  (In particular, the clusters covered by $M$ can accommodate all of $T$.)

\begin{figure}[H]
    \centering
    \includegraphics[scale=0.6]{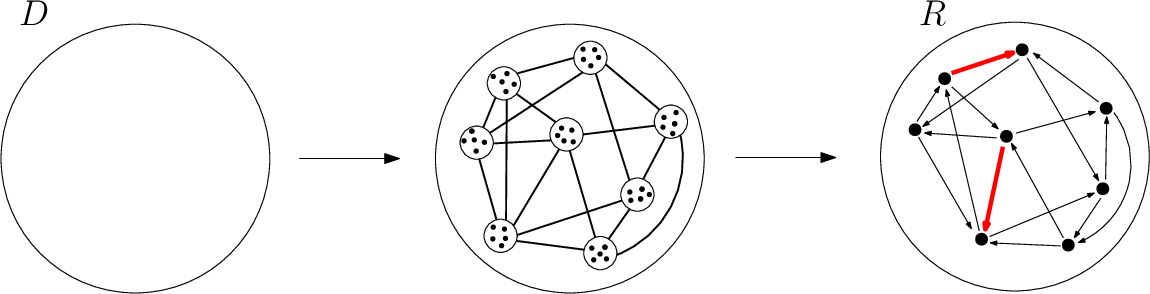}
    \caption{Regularise $D$ to obtain $R$, with a connected antimatching marked in bold.}
    \label{fig:reg}
\end{figure}

Now we turn to our antidirected tree $T$. We decompose $T$ into  a family $\mathcal{T}$ of small  subtrees (for details see Section~\ref{treedecompo}).
As we prove in Section \ref{pack}, it is possible to assign the trees in $\mathcal{T}$
 to edges of $M$ in a way that they will fit comfortably into the corresponding clusters.

\begin{figure}[H]
    \centering
    \includegraphics[scale=.72]{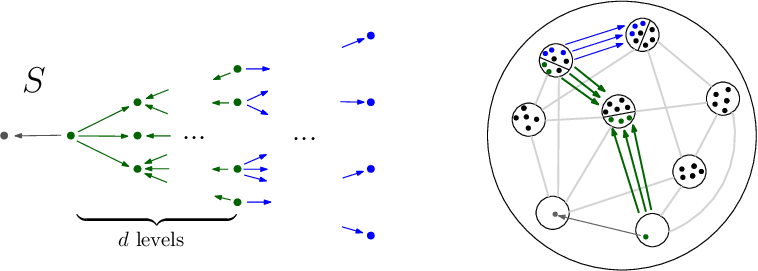}
    \caption{Embedding a small antidirected tree $S$ in $D$.}
    \label{fig:reg2}
\end{figure}
We now embed $T$ as follows. We go through the decomposition of $T$, embedding one small tree at a time, always keeping the embedded part connected in (the underlying tree of) $T$. We embed the first $d$ levels of each small tree $S$ into the clusters of an antiwalk in the reduced oriented graph that starts at the already embedded parent of the root of $S$, and leads to the edge $CD$ of $M$ that was assigned to $S$. All later levels of $S$ are embedded into clusters $C$ and $D$. 
Since $T$ has bounded maximum degree,  the union of  first $d$ levels
of  the trees in $\mathcal{T}$ is very small,  and therefore it is not a problem if  the first $d$ levels of $S$ are embedded  outside the edge $CD$. After going through all $S$, we have finished the embedding of $T$.

\paragraph{Sketch of the proof of Theorem \ref{teo:teo3}}
Let a long $k$-antisubdivision of $K_h$ be given, together with an oriented graph $D$ satisfying the conditions of the theorem. Since the antisubdivision is long, we can choose a set $\mathcal P$ of long antidirected paths of the antisubdivision such that deleting two inner vertices of each of the antidirected paths in $\mathcal P$ leaves us with an antidirected tree $T$. 

As in the proof of Theorem\ref{teo:teo1}, we find a connected antimatching $M$ in the reduced oriented graph of $D$ (see Section~\ref{sec:antiM} for details).
We embed $T$  with the help of $M$ as follows. Fixing one edge $e$ of~$M$, we embed the branch vertices of the antisubdivision into the  clusters corresponding  to $e$. We embed the long antipaths into clusters corresponding to $e$ and to other edges of $M$, with the first and last vertices on these antipaths going to clusters on  antidirected paths in the reduced oriented graph that connect the edges of $M$.

We finish the embedding by adding the two deleted vertices on the paths from~$\mathcal P$. Their neighbours were embedded into the clusters corresponding to the edge $e$, and regularity allows us to connect them by a $3$-edge path on unused vertices. This finishes the embedding of the antisubdivision.

\paragraph{Sketch of the proof of Theorem  \ref{ES}}
Given the antidirected tree $T$ and the oriented graph $D$ as in the theorem, we start by finding an oriented subgraph $D'$ of~$D$ where each vertex has either out-degree at least $\frac k2$ or out-degree $0$, and either in-degree at least $\frac k2$ or in-degree $0$. Moreover, $D'$ has at least one edge.

We then transform $D'$ into an oriented graph $D''$ which has minimum semidegree exceeding $\frac k2$. More precisely, $D''$ consists of four copies of $D'$, two of them with all edges reversed, glued together appropriately. 

We now need Theorem~\ref{teo:teo1'}, a variation of Theorem~\ref{teo:teo1} which allows us to embed~$T$ into $D''$,  guaranteeing that the root of $T$ is embedded in a small set of $V(D)$ which we can choose before the embedding. We choose to have the root  $v$ of $T$ embedded in a vertex $w$ of one of the original copies of $D$ (i.e.~a copy with edges in the original orientation). We can further ensure that one of the edges at $v$ is embedded into an edge having the original orientation of $D$. Then it is not hard to deduce that all  of~$T$  was embedded into an original copy of $D$, and therefore, $T$ is contained in $D$.

\section{Preliminaries}
\label{cap2}

\subsection{Basic digraph notation}\label{sec:basic}

A \textit{digraph} is a pair of sets $(V,E)$, where the elements of $V$ are called \textit{vertices} and $E$ is a set of ordered pairs of distinct vertices in $V$, called \textit{edges}. An \textit{oriented graph} is a digraph that allows at most one edge between a pair of vertices. 

The edges of a digraph $D$ are directed, and we will write $uv$ for an edge that is directed from $u$ to $v$. While a digraph may have both edges $uv,vu$ for any pair of vertices $u, v$, an {\it oriented graph} has at most one of these edges. The {\it out-degree} $\delta^+(v)$ of a vertex $v$ is the number of edges coming out of $v$, and the {\it in-degree} $\delta^-(v)$ is defined analogously. The \textit{minimum semidegree} $\delta^0(D)$ of a digraph is the minimum over the out-degrees and the in-degrees of all vertices in $D$.

We denote by $\vin(D)$ the set of all vertices of $D$ that have no outgoing edges, while $\vout(D)$ is the set of all vertices of $D$ that have no incoming edges. A digraph $D$ is {\it antidirected} if $V(D)=\vin(D)\cup \vout(D)$. If the underlying undirected graph of $D$ is a tree, we also call $D$ an {\it antitree}. {\it Antipaths} and {\it anticycles}  are defined analogously. Note that anticycles are even.   The {\it length} of an antipath or anticycle is the number of edges in it.

An {\it antiwalk} (or {\it antidirected walk}) is a sequence of edges that alternate directions, but we usually denote it by writing the sequence of the corresponding vertices. 
Note that an antiwalk may repeat vertices and edges, and   a repeated vertex $v$ may have both outgoing and incoming edges in the antiwalk. 
The {\it length} of an antiwalk is the length of the corresponding sequence of edges. For instance, the antiwalk $abab$ has length $3$, while having only one edge.

 Finally, an antitree $T$ is \textit{balanced} if $|V_\text{in}(T)|=|V_\text{out}(T)|$. Note that antipaths of odd length are balanced.

\subsection{Tree decompositions}\label{treedecompo}
In our proof of the main theorem, we will have to partition the given antitree into smaller trees. It will be enough to define this decomposition for  undirected trees, as we can apply it later to the underlying graph of our antitree. This type of decomposition appeared in the literature already in the 1990's, for instance in~\cite{tree}. Here we use a version from~\cite{bsp}:

\begin{definition}[$\beta$-decomposition]
Let $\beta \in (0,1)$, and let $T$ be a tree on $k+1$ vertices, rooted at $r$. If there are a set $W\subseteq V(T)$ and a family $\mathcal{T}$ of disjoint rooted trees such that
\begin{enumerate}[(i)]
    \item $r\in W$;
    \item $\mathcal{T}$ consists of the components of $T-W$, and each $S\in\mathcal{T}$ is rooted at the vertex closest to the root of $T$;
    \item $|S|\leq \beta k$ for each $S\in\mathcal{T}$; and
    \item $|W| \le \frac 1\beta + 2$,
\end{enumerate}
then the pair $(W,\mathcal{T})$ is called a \textit{$\beta$-decomposition} of $T$.
\end{definition}

\begin{lemma}$\!\!${\rm\bf\cite{bsp}}
\label{l8}
Let $\beta \in (0,1)$, and let $T$ be a tree on $k+1$ vertices, rooted at $r$. Then there is a {$\beta$-decomposition} $(W,\mathcal{T})$ of $T$.
\end{lemma}

Let us see what happens to a balanced tree $T$  if we remove some vertices from the upper levels of the trees $S\in\mathcal T$ of the tree-decomposition of $T$. We first define what exactly we wish to remove.
 As usual, the $j$th level of a tree $S$ contains every vertex in $S$ that has distance  $j$ to the root of $S$.

\begin{definition}[$Lev_m(S)$, $L_m(T)$]
Given $j\in\mathbb N$ and a $\beta$-decomposition  $(W, \mathcal{T})$ of a tree $T$,
define $Lev_j(S)$ as the union of the first $j$ levels of  $S\in\mathcal{T}$. Set  $L_j(T; (W, \mathcal{T}) )=\bigcup_{S\in\mathcal{T}}Lev_j(S)$, and write shorthand  $L_j(T)$ if $(W, \mathcal{T})$ is clear from the context.
\end{definition}
We show now that after removing $L_j(T)$ from $T$ for some bounded $j$, the remainder of $T$ is still relatively well balanced.

\begin{lemma}
\label{l9}
Let $k, j\in\mathbb{N}^+$, $\alpha, \beta\in (0, \frac 12)$. Let $T$ be a balanced rooted antitree with~$k$ edges such that $\Delta(T)\leq (\frac{\alpha\beta k}8)^{\frac 1{j+1}}$. Let $(W,\mathcal{T})$ be a  $\beta$-decomposition of $T$. Then, $(1-\alpha)p_T \leq q_T \leq  (1+ \alpha)p_T$, where
$$p_T:=|V_\text{in}(T) - L_{j}(T)|\text{ and }q_T:=|V_\text{out}(T) - L_{j}(T)|.$$
\end{lemma}

\begin{proof}
Note that $L_j(T)$ is contained in the 
union of the balls of radius $j$ centered at the  
vertices of $W$. Thus 
$$|L_j(T)|\le |W|\cdot \Delta^{j+1}\le \frac{2\Delta^{j+1}}\beta \le \frac{\alpha k}{4}.$$
As $T$ is balanced, we easily deduce the desired inequalities.
\end{proof}

\subsection{Packing trees into edges}\label{pack}

Once the given antitree is cut into small pieces $S$, and we shaved off their first levels~$L_j$, the details of which are described in  Section~\ref{treedecompo}, we will have to decide how to pack the remainder of the small pieces into the edges of the connected antimatching of the reduced oriented graph.
The following lemma shows how to allocate all of the $S\setminus L_j(T)$.

\begin{lemma}
\label{l10}
Let $m, t\in\mathbb{N}$, $\alpha>0$ and let $(p_i, q_i)_{i\in I}\subseteq \mathbb{N}^2$ be a family such that:
\begin{enumerate}[(a)]
    \item $(1-\alpha)\sum_{i\in I} p_i \leq \sum_{i\in I} q_i \leq (1+\alpha)\sum_{i\in I} p_i$,
    \item\label{balpha} $p_i + q_i \leq \alpha m$, for all $i\in I$, and
    \item \label{calpha}$\max\{\sum_{i\in I} p_i, \sum_{i\in I} q_i\} < (1-10\alpha)mt$.
\end{enumerate}
Then there is a partition $\mathcal{J}$ of $I$ of size $t$ such that for every $J\in\mathcal{J}$,
$$\sum_{j\in J} p_j \leq (1-7\alpha)m\text{ \ and \ } \sum_{j\in J} q_j \leq (1-7\alpha)m.$$
\end{lemma}

\begin{proof}
For every $i\in I$, set $\delta_i:=p_i-q_i$ and for every  $S\subseteq I$,  define $\delta_S: = \sum_{i\in S} \delta_i$. 
Let $A_1, ..., A_t\subseteq I$ be disjoint sets such that for every $j\in[t]$,
\begin{equation}\label{queulat}
     \text{$\sum_{i\in A_j} p_i \leq (1-9\alpha)m$ and $\delta_{A_j} \in [- \alpha m, \alpha m]$,}
\end{equation}
and such that, under these conditions, $R:=I\setminus (A_1\cup ...\cup A_t)$ is minimised. Note that for each  $j\in R$ there is an index $k(j)\in[t]$ such that 
\begin{equation}
\label{c14}
\text{$ p_j+\sum_{i\in A_{k_p(j)}} p_i  \leq (1-9\alpha)m$,}
\end{equation}
since otherwise, using the fact that $p_j\le\alpha m$ by~\eqref{balpha},  we have 
$$\sum_{i\in I} p_i
\ge \sum_{k\in [t]}\sum_{i\in A_k} p_i 
>
t\cdot ((1-9\alpha)m-p_j)
 \ge t\cdot (1-10\alpha)m,$$ 
 a contradiction to \eqref{calpha}. 
Next, we claim that
\begin{equation}
\label{c8}
\delta_i\delta_j \geq 0 \text{ for every }i,j\in R.
\end{equation}

To see \eqref{c8}, suppose there are $a,b\in R$ with $\delta_a\delta_b < 0$. Say $p_a\geq p_b$. Then there is an index $k(a)$ such that \eqref{c14} holds with $j=a$, and since $p_a\ge p_b$, the inequality also holds when substituting $p_a$ with $p_b$. 
Moreover, since $a\in R$,  we know that neither  $a$ nor $b$ can be added to  $A_{k(a)}$ without violating \eqref{queulat}, and thus $$\delta_{A_{k(a)}} + \delta_a, \delta_{A_{k(a)}} + \delta_b \notin [-\alpha m, \alpha m].$$
However, $\delta_{A_{k(a)}}  \in [-\alpha m, \alpha m]$ by (ii). So, since we assumed that $\delta_a\delta_b<0$, it must be that $|\delta_a|+|\delta_b|>2\alpha m$. But this contradicts~\eqref{balpha}.
 We proved \eqref{c8}, which enables us to split the rest of the proof into two cases. 

{\bf Case 1: $\delta_i\ge 0$ for all $i\in R$. }
For $k\in [t]$, we let $A'_k\supseteq A_k$ be disjoint sets such that
$\sum_{i\in A'_k} p_i \leq (1-9\alpha)m$ 
 for every $k\in[t]$, and thus also, 
 \begin{equation}\label{qq2}
\sum_{i\in A'_k} q_j = \sum_{i\in A'_k} (p_j - \delta_{A'_k}) \le \sum_{i\in A'_k} (p_j - \delta_{A_k}) \leq  (1-9\alpha)m + \alpha m < (1-7\alpha)m.
\end{equation}
Similarly as for~\eqref{c14}, we can show that for each  $j\in R':=I\setminus(A'_1\cup ...\cup A'_t)$ there is an index $k'(j)\in[t]$ with 
$ p_j+\sum_{i\in A'_{k_p(j)}} p_i  \leq (1-9\alpha)m$, and thus, by~\eqref{qq2}, $q_j+\sum_{i\in A'_{k_p(j)}} q_i  \leq (1-9\alpha)m$. 
So $R'=\emptyset$, and we are done, taking $\mathcal{J}= \{A'_1, ..., A'_t\}$. 

{\bf Case 2: $\delta_i\le 0$ for all $i\in R$.} We proceed as in Case 1, but choose sets $A'_k\supseteq A_k$  such that
$\sum_{i\in A'_k} q_i \leq (1-9\alpha)m$ 
 for every $k\in[t]$. The rest of the  argument is similar.
\end{proof}

\subsection{Diregularity}
We will use the regularity method which goes back to Szemer\'edi's work from the 1970's~\cite{reg}. We will need regularity for oriented graphs, but start introducing the concept for undirected graphs.
Let $G$ be a graph, let $\varepsilon>0$ and let $A,B$ be two disjoint subsets of $V(G)$. The \textit{density} of the pair $(A,B)$ is $d(A,B)=\frac{|E(A,B)|}{|A|\cdot|B|}$, with $|E(A,B)|$ denoting the number of edges between $A$ and $B$. The pair $(A,B)$ is $\varepsilon$-\textit{regular} if $|d(X,Y)-d(A,B)|<\varepsilon$  for every $X\subseteq A$ and $Y\subseteq B$ satisfying $|X|>\varepsilon|A|$ and $|Y|>\varepsilon|B|$. 

Let $(A,B)$ be an $\varepsilon$-regular pair of density $d$ and let $Y\subseteq B$ be such that $|Y|>\varepsilon|B|$. A vertex $x\in A$ is called $\varepsilon$-\textit{typical} (or simply \textit{typical}) with respect to $Y$ if it has more than $(d-\varepsilon)|Y|$ neighbours in $Y$. It is well-known and easy to prove that $A$ has at most $\varepsilon|A|$ vertices that are not typical with respect to $Y$. 

A partition $\{V_0, ..., V_k\}$  of $V(G)$ is called an $\varepsilon$-\textit{regular} partition of $G$ if
\begin{enumerate}[(i)]
    \item $|V_0|\leq\varepsilon|V|$,
    \item $|V_1| = ... = |V_k|$,
    \item all but at most $\varepsilon k^2$ of the pairs $(V_i, V_j)$, with $1\leq i < j\leq k$, are $\varepsilon$-regular.
\end{enumerate}

Let $D$ be a digraph and let $A,B\subseteq V(D)$ disjoint. We denote by $(A,B)$ the oriented subgraph of $D$ with vertex set $A\cup B$ and every edge directed from $A$ to $B$ in $D$. In this case, we say the pair $(A,B)$ is $\varepsilon$-\textit{regular} if the underlying graph is $\varepsilon$-regular. With these notions and based on the Szemerédi's Regularity Lemma for graphs without orientation \cite{reg}, we show its version for digraphs as stated by Alon and Shapira in \cite{direg}:

\begin{lemma}[Degree form of the Diregularity Lemma \cite{direg}]
\label{direg}
For all $\varepsilon\in(0,1)$, $m_0\in\mathbb N$, there are $M_0, n_0\in \mathbb N$ such that for each $d\in[0,1]$ and for each digraph $D$ on $n\geq n_0$ vertices there are a partition of $V(D)$ into sets $V_0, V_1, ..., V_k$ and a spanning subdigraph $D'$ of $D$, called the  \textit{regularised digraph}, such that the following holds:
\begin{itemize}
    \item $m_0 \leq k \leq M_0$,
    \item $|V_0| \leq \varepsilon n$ and $|V_1| = \dotsc = |V_k| =: m$,
    \item $d_{D'}^+ (x) > d_D^+(x) - (d+\varepsilon)n$ for all vertices $x\in D$,
    \item $d_{D'}^- (x) > d_D^-(x) - (d+\varepsilon)n$ for all vertices $x\in V(D)$,
    \item for all $1\leq i<j\leq k$ and $i\neq j$, the bipartite graph $(V_i, V_j)_{D'}$ whose vertex classes are $V_i$ and $V_j$ and whose edge set consists of all the $V_i$-$V_j$ edges in $D'$ is $\varepsilon$-regular and has density either 0 or at least $d$,
    \item for all $1\leq i \leq k$ the digraph $D'[V_i]$ is empty.
\end{itemize}
\end{lemma}

Given a regularised digraph $D'$ with clusters $V_1, ..., V_t$, the \textit{reduced digraph} $R$ is a digraph with vertices $V_1, ..., V_t$ such that the edge $V_iV_j$ exists only if $D'$ contains a $V_i$-$V_j$ edge. Observe that,  $R$ need not be an oriented graph, even if $D'$ is. However, it is possible~\cite{kko} to discard appropriate edges from the reduced digraph to find 
 an oriented spanning subgraph $R'$ of 
 $R$
 that preserves the minimum semidegree of the original oriented graph $D$ (proportionally to the order of the reduced digraph). The new oriented graph $R'$ is called the \textit{reduced oriented graph}. This result is formalised in the following lemma.

\begin{lemma}[see \cite{kko}]
\label{lreduc}
Let $\varepsilon, d\in[0,1]$, $m_0\in \mathbb N$, let $G$ be a large enough oriented graph and let $R'$ be the reduced digraph  obtained by applying the Lemma \ref{direg} to $G$ with parameters $\varepsilon, m_0$ and $d$. Then, $R'$ has a spanning oriented subgraph $R$ with
\begin{enumerate}[(i)]
    \item $\delta^+(R) \geq (\frac{\delta^+(G)}{|G|} - (3\varepsilon+d))|R|,$
    \item $\delta^-(R) \geq (\frac{\delta^-(G)}{|G|} - (3\varepsilon+d))|R|.$
\end{enumerate}
This oriented graph $R$ is called the $(\varepsilon, d)$-reduced oriented graph.
\end{lemma}

\subsection{Embedding small trees in antiwalks}

We  now show how to embed a small tree into an antiwalk of the reduced oriented graph. This will be useful for embedding the first levels of a piece $S$ of the decomposition of our antitree $T$ into an antiwalk leading to a suitable edge $a_ib_i$ of the connected antimatching. The edge $a_ib_i$ will accommodate the rest of $S$.

We start with a convenient definition,  expressing that the orientations of the small tree $S$ and the antiwalk coincide.
\begin{definition}[Consistent antiwalk and antitree]
Let $P$ be a non-trivial antiwalk and let $T$ be a rooted antitree. We say $P$ and $T$ are \textit{consistent} if  one of the following holds:
\begin{enumerate}[(i)]
    \item $r(T)\in V_{\text{out}}(T)$ and  $P$ starts with an out-edge, or
    \item $r(T)\in V_{\text{in}}(T)$ and  $P$ starts with an in-edge. 
\end{enumerate}
\end{definition}

Now we embed an antitree into a consistent antiwalk of a reduced digraph. A very similar result for undirected graphs was proved in~\cite{bsp}.

\begin{lemma}
\label{lembedding}
Let $\varepsilon\in(0, \frac{1}{4})$, $m,h\in\mathbb{N}$. Let  $P=Q_0 ... Q_{h-1}Q_{h}$ be an antiwalk in an $(\varepsilon, 2\sqrt{\varepsilon})$-reduced oriented graph $R$ of an oriented graph  $D$, where clusters have size~$m$.  For  $0\leq i\leq h$, let $Z_i\subseteq Q_i$ be such that $|Z_i|\geq 3\sqrt{\varepsilon} m $ if $i\neq 0$ and $|Z_0|\geq 3{\varepsilon} m $. For $i=h-1, h$, let $X_{i}\subseteq Q_{i}\setminus Z_{i}$  with $|X_i|>3\sqrt \varepsilon m$. Let $S$ be a rooted antitree with $|S|<\frac\varepsilon{10} m$ such that $P$ and $S$ are consistent.
Then there is an embedding $\varphi$ of $S$ into~$D$  such that
\vspace{-0.2cm}
\begin{enumerate}[(a)]
    \item for each $i\le h$, and for each $v$ from the $i$th level of $S$, $\varphi (v)\in Z_i$ and $\varphi(v)$ is typical with respect to $Z_{i+1}$ (or with respect to $Z_{i-1}$ if $i=h$), and
    \item for each $i> h$, and for each $v$ from the $i$th level of $S$, $\varphi (v)\in X_{h-1}$ is typical with respect to $X_{h}$ if $i-h$ is odd,  and $\varphi (v)\in X_{h}$ is typical with respect to $X_{h-1}$ if $i-h$ is even.
\end{enumerate}
\end{lemma}

\begin{proof}
Let $r(S)$ be the root of $S$ and denote the embedding by $\phi: V(S) \longrightarrow \cup_{i=0}^{h} Q_i$. Suppose that $Q_0Q_1\in E(R)$, and $r(S)\in V_\text{out}(S)$, the other case is  analogous.

We embed $r(S)$ in a typical vertex of $Z_0$ with respect to $Z_1$, and then embed successively, for $i< h$,  the vertices from the $i$th level into   vertices from $Z_i$ that are typical to $Z_{i+1}$.
(This is possible since we embed at most $|S|<\frac\varepsilon{10} m$ vertices in total, and since  $|Z_i|\geq 3\sqrt{\varepsilon}m$ for all $i$, while there are less than $\varepsilon m$ atypical vertices in each~$Z_i$.) For the vertices from the $h$th level we take care that their images are typical with respect to both $Z_{h-1}$ and  $X_{h-1}$. The vertices from later levels are embedded alternatingly into $X_{h-1}$ and $X_h$, and we take care that their images are typical with respect to both $Z_{h}$ and  $X_{h}$, or both $Z_{h-1}$ and  $X_{h-1}$, respectively.
\end{proof}


\section{Connected antimatchings}\label{sec:antiM}
In this section we introduce the concept of a connected antimatching, which is central to our proof. We also prove two lemmas (Lemma~\ref{l5} and Lemma~\ref{l6}) relating the minimum semidegree of an oriented graph $D$ to the existence and size of a connected antimatching in $D$, whose edges are not too distant from each other.

Before we can introduce the notion of a connected antimatching, we need some preliminary definitions.
\begin{definition}[in-vertex and out-vertex of an antiwalk]
Let $P$ be an antiwalk, and let $v\in V(P)$. 
We call $v$   an \emph{out-vertex} of $P$ if $P=v$ or $P$ has one or more edges of the form $vx$. We call  $v$  an \emph{in-vertex} of $P$ if $P=v$ or  $P$ has one or more edges of the form $xv$. 
\end{definition}
 Note that also in non-trivial walks a vertex  can be both an in- and an out-vertex.

\begin{definition}[out-walk; out-in-walk; out-out-walk]
Let $C$ be an oriented graph with $a,z\in V(C)$. A non-trivial antiwalk $P$ starting in $a$ and ending in $z$ whose first edge is directed away from $a$ is called an \emph{out-walk from $a$ to $z$}. 
Call $P$ an \emph{out-in-walk} if its last edge is directed towards $z$, and  an \emph{out-out-walk} otherwise. Also call a trivial walk an out-out-walk.
  \end{definition}

\begin{definition}[$In(C,a)$; $Out(C,a)$]
Let $C$ be an oriented graph with $a\in V(C)$. 
Then $\In(C,a)$ is the set of all  $z\in V(C)$ such that there is an out-in-walk from $a$ to~$z$, and  $\Out(C,a)$ is the set of all  $z\in V(C)$ such that there is an out-out-walk from $a$ to~$z$.  \end{definition}

Observe that $\In(C,a)$ and $\Out(C,a)$ are not necessarily disjoint.

\begin{definition}[anticonnected oriented graph]
Let $C$ be an oriented graph with $a\in V(C)$.  We say $(C,a)$ is \emph{anticonnected} if $V(C)=\In(C,a) \cup \Out(C,a)$.
\end{definition}

\begin{definition}[connected antimatching]\label{def-antimatching}
Call $M=\{a_ib_i\}_{1\leq i\leq m}$  a \emph{connected antimatching of size $m$}  in an oriented graph $D$ if $M$ is a matching in the underlying graph of $D$ and if 
$a_i\in \Out(D,a_1)$ for every $1\leq i\leq m$. We write $V(M)$ for the set of all vertices covered by $M$.
\end{definition}

Our first lemma of this section links the size of a connected antimatching in an oriented graph with its minimum semidegree. We can even choose a vertex to be included in the antimatching, as an outvertex of a matching edge (of course, by inverting all direction we could also choose it as an invertex).

\begin{lemma} 
\label{l5} Let $t\in\mathbb N^+$, let $D$ be an oriented graph with $\delta^0(D) \geq t$, and let $w\in V(D)$. Then $D$ has a connected antimatching $M=\{a_ib_i\}_{1\leq i\leq t}$ of size $t$, 
with $w=a_1$.
\end{lemma}

\begin{proof}
Let $M=\{a_ib_i\}_{1\leq i\leq m}$ be a connected antimatching of maximum size in $D$ with the property that $w=a_1$. Note that $m \geq 1$, because any single edge constitutes a connected antimatching. For contradiction, we assume that $m < t$. 
Let $C$ be the largest induced subdigraph of $D$ such that \begin{enumerate}[(i)]
    \item $M\subseteq E(C)$
   \item  $a_i\in \Out(C,a_1)$ for every $1\leq i\leq m$
   and 
    \item $(C, a_1)$ is anticonnected.
\end{enumerate}
We claim that
\begin{equation}
\label{c1}
\parbox[c]{0.9\linewidth}{if $v\in \In(C,a_1)$ then $N^-(v)\subseteq V(C)$, and if  $v\in \Out(C,a_1)$ then $N^+(v) \subseteq V(C)$.}
\end{equation}

Indeed, let $v\in \In(C,a_1)$ and suppose that there is a vertex $x\in N^-(v)\setminus V(C)$. By definition of $\In(C,a_1)$, the digraph $C$ contains an out-in-walk $P_v$ from $a_1$ to $v$. Then $P_v$ and the edge $xv$ form an out-out-walk from $a_1$ to $x$. Thus, adding the vertex $x$ and the edge $xv$ to $C$,  we obtain a subdigraph $C'$ of $D$ that is larger than $C$ and fulfills
(i)-(iii), a contradiction. The proof for $v\in \Out(C,a_1)$ is analogous. This proves \eqref{c1}. 

Now, we show that
\begin{equation}
\label{c2}
\parbox[c]{0.9\linewidth}{if $v\in \In(C,a_1)\setminus V(M)$ then $N^-(v)\subseteq V(M)$, and if $v\in \Out(C,a_1)\setminus V(M)$ then $N^+(v) \subseteq V(M)$.}
\end{equation}

In order to see \eqref{c2}, let $v\in \In(C,a_1)\setminus V(M)$, and let $x\in N^-(v)$. By \eqref{c1}, we know that $x\in V(C)$. For contradiction, assume $x\notin  V(M)$. 
Define $a_{m+1} := x$, $b_{m+1} := v$ and set $M':=\{a_ib_i\}_{1\leq i\leq m+1}$. Consider an out-in-walk from $a_1$ to $v$ and add the edge $xv$ at the end to see that $x\in \Out(C,a_1)$. So $M'$ is a connected antimatching of size $m+1$, with $w=a_1$, a contradiction to the choice of $M$. 
We can proceed similarly for $v\in \Out(C,a_1)\setminus V(M)$. This proves \eqref{c2}.

Next, we claim that
\begin{equation}
\label{justone}
|V(C)\setminus V(M)| \leq 1.
\end{equation}
To prove \eqref{justone} by contradiction, we start by considering a vertex $u$
in $V(C)\setminus V(M)$. Since $\delta^0(D) \geq t$ and because of \eqref{c2}, we know that if $u\in \In(C,a_1)$, then $u$  has at least $t$ in-neighbours in $V(M)$, and if $u\in \Out(C,a_1)$, then $u$ at least $t$ out-neighbours in $V(M)$. Now suppose there are two distinct vertices in $V(C)\setminus V(M)$. Since we assume that $|M|<t$,  there are  distinct $u_1,u_2\in V(C)\setminus V(M)$,  edges $e_1, e_2\in E(C)$, an index $j\le t-1$ and an edge $v_1v_2:=a_jb_j\in M$ such that  for each $i=1,2$ one of the following holds:
\begin{itemize}
    \item $u_i\in \In(C,a_1)$ and $e_i=v_iu_i$, or
    \item $u_i\in \Out(C,a_1)$ and $e_i=u_iv_i$, or
\end{itemize}

If possible, choose $j\neq 1$, and if that is not possible try to choose $u_1$ in a way that $e_1=a_1u_1$.

Obtain $M'$ from $M$ by adding $e_1$ and $e_2$ to $M$ and removing $e_3$. Note that for $i=1,2$, if $e_i=u_iv_i$ then $u_i\in \Out(C,a_1)$, and furthermore, if $e_i=v_iu_i$ then $u_i\in \In(C,a_1)$, which means that $u_i$ is reachable  by an out-in-walk from $a_1$, to which we can add the edge $v_iu_i$ to obtain an out-out-walk from $a_1$ to $v_i$, implying that $v_i\in \Out(C,a_1)$. 
So, in the case that $j\neq 1$, 
and in the case that $j=1$ and $e_1=v_1u_1$,
we know that $M'$ is  a  connected antimatching which is larger than $M$, with $w$ being the outvertex of a matching edge we can call $a_1b_1$. This is a contradiction to the choice of $M$. Thus $j=1$ and $e_1=u_1a_1$.

Note that the fact that $j=1$ implies that  both $u_1$ and $u_2$ have only $t-2$ out- or in-neighbours  in $V(M)\setminus\{a_1, b_1\}$, where we count out-neighbours for $u_i$ if $u_i\in \Out(C,a_1)$ and in-neighbours otherwise. Thus both are connected by edges to both of $a_1, b_1$. Since we would have chosen $e_1=a_1u_1$ if that was possible for a choice of $u_1$, we know that both $u_1$, $u_2$ are in-neighbours of $a_1$, in fact, $a_1$ has no out-neighbours outside $V(M)$. So, because of our bound on the minimum semidegree,
$a_1$ has at least $t-1$ out-neighbours in $V(M)\setminus\{a_1, b_1\}$.
As the outdegree of $u_1$ into $V(M)\setminus\{a_1, b_1\}$ is at least $t-2$,   we  find an edge $a_jb_j\in M$, with $j\neq 1$, such that $a_1a_j, u_1b_j\in E(D)$ or $a_1b_j, u_1a_j\in E(D)$. Adding these two edges and $e_2$ to $M$ while deleting $a_1b_1$ and $a_jb_j$ gives a connected antimatching $M''$, which is larger than $M$, and where $a_1=w$ (choosing the corresponding edge as $a_1b_1$).
Therefore, inequality~\eqref{justone} holds. 

By \eqref{justone}, and since we assume $M$ to have at most $t-1$ edges, we see that
\begin{equation}
\label{c3}
|V(C)|<2t.
\end{equation}

Now, by \eqref{c1}, the out-neighbourhood of $a_1$ is contained in $V(C)$, and thus is a subset of $\In(C,a_1)$. Similarly, the in-neighbourhood of $b_1$ is a subset of $\Out(C,a_1)$. Considering the minimum semidegree of $D$, we deduce that $$|\In(C,a_1)|, |\Out(C,a_1)| \geq t.$$ By \eqref{c3}, we conclude $\In(C,a_1) \cap \Out(C,a_1) \neq \emptyset$ and thus contains  a vertex $v$. 
Because of~\eqref{c1}, both the in-neighbourhood and the out-neighbourhood of $v$ are contained in $C$. As each of these two disjoint sets has at least $t$ elements,  $C$ has at least $2t$ vertices, a contradiction to \eqref{c3}.

\end{proof}

For the next lemma, we need another definition.
\begin{definition}[out-out-distance $\dist(a,a')$; out-in-distance $\disti(a,a')$]
In an oriented graph $D$,  the \emph{out-out-distance} $\dist(a,a')$ between two vertices $a, a'\in V(D)$ is the length of the shortest out-out-walk from $a$ to $a'$, if such a walk exists. Otherwise  $\dist(a,a')=\infty$.\\
Similarly, the \emph{out-in-distance} $\disti(a,a')$ between two vertices $a, a'\in V(D)$ is the length of the shortest out-in-walk from $a$ to $a'$, if such a walk exists, and otherwise  $\disti(a,a')=\infty$.
\end{definition}
Observe that $\dist(a,a)=0$ and for edges $ab,a'b'$ of   a connected antimatching,  $\dist(a,a')$ is finite by definition.
For use in the proof of the following lemma, we define 
$\dist(M)=\sum_{i=2}^{d} \dist(a_1, a_i)$, where  $M = \{a_ib_i\}_{1\leq i\leq t}$ is a connected antimatching.

\begin{lemma}
\label{l6}
Let $t\in\mathbb N$ and let $D$ be an oriented graph with $\delta^0(D)\geq t$. Let $w\in V(D)$. Then $D$ contains a connected antimatching $M = \{a_ib_i\}_{1\leq i\leq t}$, with $w=a_1$, and such that for every $1\leq i\leq t$, $\dist(a_1, a_i) \leq 8t$.
\end{lemma}

\begin{proof}
By Lemma \ref{l5}, we know that $D$ contains a connected antimatching $M'=\{a'_ib'_i\}_{1\leq i\leq t}$  with $w=a'_1$. Among all such antimatchings  choose $M=\{a_ib_i\}_{1\leq i\leq t}$   such that 
$\dist(M)$ 
is minimised.
For the sake of contradiction, suppose there exists~$k$ with $2\le k\le t$ such that $\dist(a_1, a_k)> 8d$. Let $P$ be 
a shortest out-out-walk  from $a_1$ to $a_k$. We claim that
\begin{equation}
\label{c4}
\parbox[c]{0.9\linewidth}{each  $v\in V(D)$ appears at most once as an in-vertex and at most once as an out-vertex on $P$.}
\end{equation}
Indeed, suppose vertex  $v$  appears at least twice as an in-vertex  on $P$. Then there are distinct vertices $x,y\in V(D)$ such that  $P=P_1xvP_2yvP_3$ (where the each of paths $P_i$ is allowed to be empty). The antiwalk $P_1xvP_3$ is shorter than $P$, a contradiction to the choice of $P$. We can argue similarly if $v$ appears twice on $P$ as an out-vertex. This proves \eqref{c4}. 

From \eqref{c4} it follows that $P$ does not repeat edges. Therefore, and since we assumed that $P$ has length greater than $8t$, we know that
\begin{equation}
\label{c5}
|E(P)| > 8t.
\end{equation}

By \eqref{c4}, every $a_i$ is incident to at most 4 edges in $P$. The same holds for every~$b_i$. So at most $8t$ edges have one of their extremes on $M$, and hence, by \eqref{c5}, there is an edge $xy$ on $P$ such that $x,y\notin V(M)$. 
Replacing $a_kb_k$ with $xy$ in $M$, we obtain a connected antimatching $M'$ of size $t$ with $\dist(M')<\dist(M)$, a contradiction to the choice of $M$.\qedhere
\end{proof}


\section{Antitrees: The proof of Theorem \ref{teo:teo1}}
\label{proofteo}

Instead of Theorem \ref{teo:teo1}, we will prove a slightly stronger result, namely Theorem \ref{teo:teo1'} below.
The additional properties of Theorem \ref{teo:teo1'} will be necessary for the proof of Theorem~\ref{ES}.

Let us define the following shorthand notation.

\begin{definition}\label{gamma-emb}
For digraphs $A$ and $D$, we write $A\subseteq_\gamma D$
if for each set $V^*\subseteq V(D)$ of size at least $\gamma |V(D)|$ and for each $x\in V(A)$, there is an embedding of $A$ in $D$ with $x$ mapped to $V^*$.
\end{definition}

Here is the result that immediately implies  Theorem \ref{teo:teo1}.
\begin{theorem}
\label{teo:teo1'}
For all $\eta\in(0,1)$, $c\in \mathbb N$ there is $n_0$ such that for all $n\geq n_0$ and $k\geq \eta n$ the following holds for every oriented graph $D$ on $n$ vertices 
and every  balanced antidirected tree $T$ with $k$ edges.
If $\delta^0 (D)>(1 + \eta)\frac k2$ and $\Delta(T)\leq (\log(n))^c$, then 
$T\subseteq _\eta D$.
\end{theorem}

In the remainder of this section, we prove Theorem \ref{teo:teo1'}. 

\begin{proof}[Proof of Theorem \ref{teo:teo1'}] We will first define our constants, then prepare the given digraph $D$ and the antitree $T$, and finally proceed to embed $T$ in $D$.

\paragraph{Setting the constants.}
We define our constants $\varepsilon, \beta$ and $n_0$ so that $$\frac {1}{n_0} \ll \beta \ll \varepsilon \ll \eta<1.$$
More precisely, given $\eta$ and $c$, we set $\varepsilon := \frac{\eta^2}{10^5}$. Lemma \ref{direg} with input $\varepsilon$ and $m_0:=\lceil \frac 1\varepsilon\rceil$  gives constants $n'_0$ and $M_0$ such that we can apply the lemma to digraphs on $n\geq n'_0$ vertices.
Set $\beta:=\frac{\varepsilon}{100 M_0}$, and choose $n_0\ge n'_0$ such that 
\begin{equation}\label{n0nice}
(\log n_0)^{18 M_0c}\leq  \frac{\beta^3}{10}n_0.
\end{equation}
Finally, 
let  $n\geq n_0$ and let $k\geq\eta n$. 

\paragraph{Objective.} 
Let $D$ be an oriented graph $D$ on $n$ vertices, with $\delta^0 (D)>(1 + \eta)\frac k2$. Let $T$ be a balanced antidirected tree $T$ with $k$ edges, with $\Delta(T)\leq (\log(n))^c$. 
We need to show that 
 for every  every $V^*\subseteq V(D)$ with $|V^*|\ge\eta n$, the antitree $T$ can be embedded in $D$, with $x$ embedded in $V^*$.

So let such $V^*\subseteq V(D)$ and $x\in V(T)$ be given. 
Note that we can assume that $x\in \vout(T)$, as otherwise we could switch the orientations of $D$ and of $T$, which would move $x$ to $\Out(T)$. Once $T$ is embedded, we switch all orientations back to normal.

\paragraph{Preparing the oriented graph $D$.}
We apply Lemma \ref{direg} with  $d=2\sqrt{\varepsilon}$ to obtain~$D'$, a digraph with $r\leq M_0$ clusters $C_1, ..., C_r$ of the same size $m$, and with $\delta^0(D') > (1+\frac \eta2)\frac k2$. We divide each cluster of $D'$ into two slices, $C_i^1$  and $C_i^2$ of sizes 
\begin{equation}
\label{slices}
|C_i^1|=\lfloor 10\sqrt{\varepsilon} |C_i|\rfloor\text{ and } |C_i^2|=\left\lceil(1- 10\sqrt{\varepsilon})|C_i|\right\rceil.
\end{equation}

Note that at least one of the clusters $C_i$ contains at least $\eta|C_i|$ vertices from $V^*$. Let $C^*$ be such a cluster.

Let $R$ be the reduced oriented graph of $D'$ given by Lemma \ref{lreduc}, on vertices $C_1, ..., C_r$, of minimum semidegree greater than $$t:=\lceil(1+\frac{\eta}{2})\frac{kr}{2n}\rceil.$$ By Lemma \ref{l6}, $R$ has a connected antimatching $M=\{a_ib_i\}_{i=1}^t$ such that 
\begin{equation}
\label{c12}
\text{$\dist(a_1, a_j)\leq 8t$ for each $j\in[t]$,}
\end{equation}
and such that $C^*=C_{a_1}$,
where here and later we  write $C_{a_i}$ (resp.~$C_{b_i}$) for the cluster corresponding to $a_i$ (resp.~$b_i$) in our reduced oriented graph $R$, for every antimatching edge $a_ib_i\in M$.

\paragraph{Preparing the antitree $T$.}
Note that because of~\eqref{n0nice}, and since we assume that $\Delta(T)\le\log n$, 
\begin{equation}\label{knice}
k\ge \eta n\ge 10\beta^{-2}(\Delta(T))^{18M_0}.
\end{equation}

We root $T$ at the $x$.
By Lemma \ref{l8}, there is a $\beta$-decomposition $(W,\mathcal{T}')$ of the underlying undirected tree of $T$, with $|W|<\frac 1\beta + 2\le\frac 2\beta$. Let $\mathcal{T}$ denote the set of oriented subtrees of $T$ corresponding to $\mathcal{T}'$. For each $S\in\mathcal{T}$, we define 
$$(p_S, q_S):=\big(|V_\text{in}(S) - Lev_{16t+2}(S)|, |V_\text{out}(S) - Lev_{16t+2}(S)|\big ).$$

We claim that the family $(p_S, q_S)_{S\in\mathcal{T}}$ satisfies the conditions of Lemma \ref{l10} with  $\sqrt{\varepsilon}$ playing the role of $\alpha$. Indeed, Lemma \ref{l10}$(a)$ holds by Lemma \ref{l9},  because of~\eqref{knice}, and since $t\ge M_0$.
As $(W,\mathcal{T}')$ is a $\beta$-decomposition, and by our choice of $\beta$, we know that $$p_S+q_S \leq |S| \leq \beta k \leq \beta n  \leq \frac{\sqrt{\varepsilon}}2m\le \sqrt\varepsilon |C_{a_1}^2|$$
for each $S\in\mathcal T$, and thus,  Lemma \ref{l10}$(b)$ holds. Finally Lemma \ref{l10}$(c)$ is true because $T$ is balanced, and thus, by our choice of $t$, and since $|W|\ge 1$,
\begin{align*}
\max\{\sum_{S\in\mathcal{T}} p_S, \sum_{S\in\mathcal{T}} q_S\} 
&\leq \frac{k}{2}
\le (1-10\sqrt\varepsilon)^2(1-\varepsilon)(1+\frac\eta 2)\frac k2 \\
&\le (1-10\sqrt\varepsilon)^2(1-\varepsilon)\frac n{r} t \le (1-10\sqrt\varepsilon)|C_{a_1}^2|t.   
\end{align*}
So Lemma \ref{l10}, gives a partition $\{P_j\}_{j=1}^{t}$ of $\mathcal{T}$, such that for every $j\in[t]$, 
\begin{equation}
\label{c15}
\max \{\sum_{i\in P_j} p_i, \sum_{i\in P_j} q_i \}\leq (1-7\sqrt{\varepsilon})|C_{a_1}^2|.
\end{equation}

For convenience, let us also define, for each $S\in\mathcal T$, the set $W_S$. This set contains all vertices from $W$ whose path to $r(T)=x$ passes through $S$ before passing through any other $S'\in\mathcal T$. That is, $W_S$ contains all children of $S$ in $W$, as well as their children in $W$, etcetera. Similarly, we define $W_{r(T)}$ as the set of all vertices from $W$ whose path to $r(T)$ does not meet any tree from $\mathcal T$.

\paragraph{Idea of the embedding procedure.}
The process starts with embedding $r(T)=x$ into  $C^*$. Then we embed the vertices from $W_{r(T)}$ into $C_{a_1}^1\cup C_{b_1}^1$.

After that, in every step of the process, we will embed some $S\in\mathcal{T}$, say with $S\in P_j$, together with the set $W_S$. We choose $S$ such that  the parent of its root $r(S)$ is already embedded, say in  cluster~$C$. As $C$ itself or a neighbouring cluster $C'$ lies at out-out-distance at most $8t$ to $C_{a_1}$ in $R$, 
there is a short antiwalk $P$ from $C$ or from $CC'$ to $C_{a_j}$. We  embed   the first $16t+2$ levels of $T[S\cup W_S]$ into the $Q^1$-slices of clusters $Q$ of $P$, and the rest into $C_{a_j}^2, C_{b_j}^2$. 

The bound on the maximum degree of $T$ ensures that only few vertices in total go to clusters from connecting antiwalks $P$. Therefore, $T$ will be mainly embedded into the edges from $M$, where we have control on how space is allocated.

\paragraph{The embedding.}
Let us make this sketch more precise. Let $V_1 = \bigcup_{i=1}^r C_i^1$ and $V_2 = \bigcup_{i=1}^r C_i^2$. Embed $r(T)=x$ in a vertex of $C^*\cap W=C_{a_1}\cap W$ that is typical to $C_{b_1}^1$. 
We next embed all of $W_{r(T)}$ into $C_{a_1}^1\cup C_{b_1}^1$. This can be done levelwise, in each step choosing vertices that are typical with respect to $C_{a_1}^1$ or $C_{b_1}^1$, respectively. As $W_{r(T)}$ has at most $\frac 2\beta $ vertices, we can embed all of $W_{r(T)}$ without a problem.

We now go through the antitrees $S\in \mathcal{T}$, and embed 
$S$ together with  $W_S$  in $ W$. We do this in an ordered way, so that when starting to work with $S$, the parent of $r(S)$ is already embedded.  
We will show that at every step of the process, i.e.~for each $S\in\mathcal{T}$,  the following conditions are met:
\begin{enumerate}[(a)]
\item \label{a}  $Lev_{16t+2}(T[S\cup W_S])$ is embedded into $V_1$, and the rest of $T[S\cup W_S]$ is embedded  into $V_2$,
\item \label{b} if $S\in P_j$, then $V_\text{out}(T[S\cup W_S])\setminus Lev_{16t+2}(T[S\cup W_S])$ is embedded into $C_{a_j}^2$ and $V_\text{in}(T[S\cup W_S])\setminus Lev_{16t+2}(T[S\cup W_S])$ is embedded into $C_{b_j}^2$, 
\item \label{c} every vertex in  $V_\text{out}(T[S\cup W_S])$ is embedded in a cluster $C_v$ corresponding to a vertex $v$ with $\dist (a_1,v)\le 8t$,
\item \label{c'} every vertex in $V_\text{in}(T[S\cup W_S])$ is embedded in a cluster $C_v$ corresponding to a vertex $v$ with $\disti (a_1,v)\le 8t$, and
\item \label{d} for every $w\in W$, if the image of $w$ lies in cluster $C$ then it is typical  with respect to  $(C')^1$ for some neighbour~$C'$ of $C$.
\end{enumerate}

Assume now we  are in a step of the process and about to embed $T[S\cup W_S]$ for some $S\in \mathcal{T}$. 
Let $j\in[t]$ be such that $S\in P_j$. Observe that the parent of $r(S)$ is already embedded in a  vertex $w$ of a cluster $C$ so that $w$ is typical to $(C')^1$ for some neighbour~$C'$ of cluster~$C$, by \eqref{d}. 

First assume that  $r(S)\in V_\text{out}(S)$. 
Then the parent of $r(S)$ is an in-neighbour of $r(S)$.  Because of \eqref{c} and \eqref{c12}, there is  an out-out-walk $P''$ of (even) length at most $16t$ starting at $C$ and ending at $C_{a_j}$. We add the cluster $C_{b_j}$ at the end and the cluster $C'$ at the beginning to obtain an antiwalk $P'$ of length at most $16t+2$. We obtain an antiwalk $P$ of length exactly $16t+2$ by repeating two subsequent vertices of $P'$ an appropriate number of times. Note that $P$ and $T[S\cup W_S]$ are consistent. 

Now, if $r(S)\in V_\text{in}(S)$, we proceed analogously, only that here, $P''$ is an in-in-walk of length at most $16t-1$, and when constructing $P'$, we add the antiwalk $C_{b_j}C_{a_j}$ at the end. That gives an in-out-walk $P$ of length exactly $16t+2$ that is consistent with $T[S\cup W_S]$. 

Observe that $|S\cup W|<\frac\varepsilon{10} m$.
Set $h:=16t+2$,
and let $X_{h-1}$ and $X_{h}$ be the sets of all unoccupied vertices of the last two clusters on $P$. By \eqref{a}, the only vertices embedded in $C_{a_j}^2\cup C_{b_j}^2$ are vertices from $W$ and antitrees from $P_j$ without their first $16t+2$ levels. By \eqref{slices} and \eqref{c15}, we conclude that 
 \begin{equation}\label{enough}
 |X_{h-1}|, |X_h|>3\sqrt{\varepsilon} m. 
\end{equation}

We will use  Lemma \ref{lembedding} to embed $T[S\cup W_S]$  into $P$. For this, let $Z_0$ be the set of all unused out-neighbours  (if $r(S)\in V_\text{out}(S)$) or in-neighbours (if $r(S)\in V_\text{out}(S)$)  of $w$ in $C'$. For $1\le i\leq h$, let $Z_i$ be the set of all  unused vertices of $Q^1$, where $Q$ is the $(i+1)$th vertex on $P$. By \eqref{a}, any used vertex in $Q^1$ lies in $L_{16t+2}(T)\cup W$, and this set contains at most
$$ |W|\cdot \Delta(T)^{18t} \leq \frac 2\beta \cdot (\Delta(T))^{18  r\frac kn} \leq \frac 2\beta \cdot(\Delta(T))^{18 M_0}
\le \frac{\beta }5 k
\le \frac {\varepsilon}{500 M_0} n
    \leq \varepsilon m$$
vertices, where we used~\eqref{knice} for the third inequality. So,  since by \eqref{slices}, 
 the neighbourhood of $w$ in $(C')^1$ is at least $d\cdot |(C')^1|\ge 2\sqrt\varepsilon\cdot 9\sqrt\varepsilon|C'| \ge 18 \varepsilon m$, we conclude that 
$|Z_1|>3{\varepsilon}m$ and similarly, we see that 
$|Z_i|>3\sqrt{\varepsilon}m$ for $1\le i\leq h$. Thus, we can apply Lemma \ref{lembedding} to embed   $T[S\cup W_S]$. Observe that conditions \eqref{a}--\eqref{d} are satisfied after embedding $S$.

After embedding $T[S\cup W_S]$ for each $S\in\mathcal T$ in this way, we have embedded all of~$T$, with $x$ embedded  in $V^*$, which finishes the proof.
 \end{proof}

\section{Antisubdivisions of $K_h$: The proof of Theorem~\ref{teo:teo3}}
\label{sec:subdi}

The proof of Theorem \ref{teo:teo3} is  similar to the proof of Theorem \ref{teo:teo1}, with the difference that we only need the connected antimatching in the reduced oriented  graph, and then do the embedding `by hand'. 

\paragraph{Preparation.} We choose our constants as in the proof of Theorem \ref{teo:teo1}, in particular ensuring that $n_0\ge 100\varepsilon^{-1}M_0^3$. In addition, we set 
$$\gamma:=\frac{\varepsilon}{10M_0}.$$
We now prepare the oriented graph  in the same way as for Theorem \ref{teo:teo1}. In the reduced oriented graph of $D$, we use Lemma~\ref{l5} to find a connected antimatching $M=\{a_ib_i\}_{1\le i\le t}$ of convenient size $t$. 
We split each cluster $C$ into two slices $C^1$ and~$C^2$, where $C^1$ only contains about a $10\sqrt\varepsilon$ portion of the vertices of $C$.

Let a long $k$-edge antisubdivision $H$ of $K_h$ be given, that is, a set $X=x_1,\ldots, x_h$ of vertices, and antipaths $P_{i,j}$ connecting $x_i$ with $x_j$, for each $1\le i<j\le h$. For each two-edge path $P_{i,j}$ we add the middle vertex $x_{i,j}$ to $X$. Call the obtained set~$X'$. By assumption, $X'$ induces a forest in $K_h$. Choose  a subset of $$\{P_{i,j} : 1\le i<j\le h \land e(P_{i,j})\ge 3\}$$ so that together with $X'$, they induce a tree   in~$K_h$. For all other paths $P_{i,j}$
we choose a subpath $Q_{i,j}\subseteq P_{i,j}$ of length $3$. Let $\mathcal Q$ be the set of all such paths $Q_{i,j}$, and let $Y$ be the set of their endvertices. Note that $T=K_h-(\bigcup_{Q_{i,j}\in \mathcal Q}V(Q_{i,j})\setminus Y)$ is an antitree. We root $T$ at $x_1$. 

Let $Z$ be the set of all vertices on paths in $T$ between vertices of $X\cup Y$ that have length at most $M_0$.
Set $W=X'\cup Y\cup Z$, and note that
 $W$ naturally partitions into two sets: $W_{in}=W\cap \vin(H)$ and $W_{out}=W\cap \vout(H)$. Observe that there are at most $\frac{h^2}2$ paths $P_{i,j}$ of length at least $3$, and thus 
\begin{equation}\label{Wsmall}
|W|\le 2h+h^2+M_0h^2\le 4M_0h^2.
\end{equation}
 Also note that
 each component of $T- W$ is a long path (longer than $M_0$). Let $\mathcal P$ be the set of all these components. 

\paragraph{Embedding of $T$.}
We first give quick overview of the embedding of $T$, the details are given further below. First, we embed the root $x_1$ into $C_{a_1}\cup C_{b_1}$. Then  at each step, we embed either a vertex from $W$ or a path from $\mathcal P$ whose parent is already embedded. All vertices of $W$ are embedded into  $C_{a_1}^1\cup C_{b_1}^1$, and
the paths from $\mathcal P$ are embedded mainly into the $C^2$-slices of clusters corresponding to edges of $M$. For this, we use the edges of the antimatching $M$ (including $a_1b_1$ although this is not crucial) in an ordered way, and when one edge is sufficiently used, we go to the next edge of $M$. For the connections between $C_{a_i}\cup C_{b_i}$ and $C_{a_j}\cup C_{b_j}$ we use the slices $C^1$. At all times, vertices are embedded into typical vertices with respect to the unused part of the slice $C^2$ to be used in the next step, or with respect to the slice $C^1$ to be used in the future for children.
The details of the embedding of $W$ are given in the next paragraph, and the details for the paths from $\mathcal P$ are given in the subsequent paragraph.

The vertices of $W$ are embedded into $C_{a_1}^1\cup C_{b_1}^1$, as mentioned  above. 
Namely, each $w\in W$ (in particular $x_1$) is embedded into $C_{a_1}^1$ if $w\in W_{out}$ and $w$ is embedded into $C_{b_1}^1$ if $w\in W_{in}$. Every time we embed a vertex from $W$ into $C_{a_1}^1$, we make sure its image is typical   with respect to $C_{b_1}^1$, and every time we embed a vertex from $W$ into $C_{b_1}^1$, we make sure its image is typical  with respect to $C_{a_1}^1$. By~\eqref{Wsmall}, the set $W$ is small enough to easily fit into $C_{a_1}^1\cup C_{b_1}^1$. It is also much smaller than a typical neighbourhood in $C_{a_1}^1$ or in $C_{b_1}^1$, and therefore, it is not a problem if a vertex from $W$ is embedded much earlier than some of its children (as long as we ensure the children are embedded into $C^1$-slices). For a more precise analysis, see below.

Let us now turn to
the paths from $\mathcal P$. As mentioned above, these paths are embedded mainly into the $C^2$-slices of clusters corresponding to edges of the antimatching $M$. However, the first vertex $s$ of such a path $S$ goes to either $C_{a_1}^1$ or  $C_{b_1}^1$ (depending on the image of its parent), and we choose for the image of $s$ a vertex that is typical with respect to $C_0^1$ where $C_0$ is the first cluster on a path $P$ from $C_{a_1}$, or from  $C_{b_1}$, to the clusters $C_{a_j}$ and $C_{b_j}$ we plan to use for $S$. We embed the next vertices of $S$ along $P$, always into $C^1$-slices, always typical to the next slice until we reach $C_{a_j}\cup C_{b_j}$ (after at most $M_0$ steps). Then we switch to the slices $C_{a_j}^2$ and $C_{b_j}^2$. Now, every  vertex is embedded into a typical vertex with respect to the unused part of  $C_{a_j}^2$ or of $C_{b_j}^2$. If necessary, we move from $C_{a_j}\cup C_{b_j}$ to $C_{a_{j+1}}\cup C_{b_{j+1}}$ (using at most~$M_0$ vertices which are embedded into $C^1$-slices on a suitable path). We keep moving to subsequent  $C_{a_i}\cup C_{b_i}$'s if necessary. When we reach the last  vertices on~$S$, we start moving back to $C_{a_1}\cup C_{b_1}$, through the $C^1$-slices on a path of length at most~$M_0$. We reach $C_{a_1}\cup C_{b_1}$ and embed the last vertex of $S$ into a typical vertex with respect to $C_{a_1}^1$ or to $C_{b_1}^1$. If $S$ has a child in $W$, this child can be embedded now, or later in the process. 

 Let us analyse the embedding procedure to see that $T$ can indeed be embedded in this way. First, we note that the size of $M$ is sufficient, and the $C^2$-slices are large enough so that all the paths from $\mathcal P$ can be embedded without a problem. Second, we observe that we  used the $C^1$-slices exclusively for $W$, for the first $\le M_0$ and last $\le M_0$ vertices on a long path, and for moving from one edge $a_jb_j$ to the next edge $a_{j+1}b_{j+1}$. 
 So the number of vertices embedded into the $C^1$-slices is at most
$|W|+2M_0\cdot h^2+M_0\cdot t$. Because of~\eqref{Wsmall} and since $t\le M_0$, this number is bounded from above by 

$$6M_0h^2 + M_0^2 \le\varepsilon \frac n{8M_0}\le\frac\varepsilon 2 m.$$

So, whenever are about to embed a vertex into a $C^1$-slice, it is enough to know that the image of its parent was chosen typical with respect to the whole slice (as a typical neighbourhood in the slice has at least $18\varepsilon m$ vertices). We can thus embed all of $T$.

\paragraph{Embedding interior vertices on paths from $\mathcal Q$.}It only remains to embed the two middle vertices $q_{i,j}$, $q'_{i,j}$ of the paths $Q_{i,j}$, which we do successively, in any order. Say the neighbour of $q_{i,j}$ was embedded in $v\in C^1_{a_1}$, and the neighbour of  $q'_{i,j}$  was embedded in $v'\in C^1_{b_1}$. We consider the neighbourhood $N_1$ of the image of $v$ in the unused part of $C^1_{b_1}$ and the neighbourhood $N_2$ of the image of $v'$ in the unused part of $C^1_{a_1}$. Since at most $ \frac\varepsilon 2 m+h^2\le \varepsilon  |C^1_{a_1}|$ vertices of $C^1_{a_1}$ and $C^1_{b_1}$ have been used so far, we know that $N_1$ and $N_2$
 are sufficiently large  to ensure that there is an edge $ww'$ we can use. We complete the embedding by mapping $v$ and $v'$ to $w$ and $w'$.

\section{Edge density: The proof of Theorem \ref{ES}}\label{edge-dens}

In order to prove Theorem \ref{ES} we will need a  lemma that allows us  to rewrite the condition on the edge density as a condition on a parameter that is very similar to the semidegree, and which we will call the minimum pseudo-semidegree. Define the {\it minimum pseudo-semidegree 
$\bar\delta^0(D)$ }
of a  digraph $D$ with at least one edge as the minimum $d$ such that for each vertex $v\in V(D)$ we have $\delta^+(v), \delta^-(v)\in \{0\}\cup [d,\infty)$. The minimum pseudo-semidegree 
of an empty  digraph is $0$.
\begin{lemma}\label{pseu}
Let $k\in \mathbb N^+$. If a digraph $D$ has more than $(k-1)|V(D)|$ edges, then it contains a digraph $D'$ with $\bar\delta^0(D')\ge \frac k2$.
\end{lemma}
\begin{proof}
Note that the vertices of $D$ have, on average, in-degree greater than  $k-1$ and out-degree  greater than  $k-1$. Consider the following folklore construction of an auxiliary bipartite graph~$B$ associated to the digraph $D$: first,  divide each vertex $v\in V(D)$ into two vertices $v_{in}$ and $v_{out}$, letting $v_{in}$ be adjacent to all edges ending at $v$, and letting~$v_{out}$ be adjacent to all edges starting at $v$; second, omit all directions on edges. 

Then the average degree of $B$ is greater than   $k-1$, and a standard argument shows that~$B$ has a non-empty subgraph $B'$ of minimum degree at least $\frac{k}2$ (for this, it suffices to successively delete edges of degree at most $\frac{k-1}2$ and to calculate that we have not deleted the entire graph). 
Translating $B'$ back to the digraph setting, we see that $D$ has a subdigraph~$D'$, such that $D'$ minimum pseudo-semidegree at least~$\frac{k}2$, which is as desired.  
 \end{proof}

The next auxiliary result states that the minimum semidegree and the minimum pseudo-semidegree are practically equivalent for the purposes of finding antidirected subgraphs $A$ in oriented graphs $D$, if there is some control over the placement of $A$ in~$D$. 

Say an oriented graph is {\it weakly connected} if its underlying graph is connected. Recall that the notation $\subseteq_\gamma$ is given  Definition~\ref{gamma-emb}.

\begin{lemma}\label{antilemma} For any weakly connected antidirected graph $A$ and for any
$\ell, n_0\in \mathbb N$, the following holds.
If for each oriented graph $D$ on at least $n_0$ vertices with $\delta^0(D)\ge \ell$ we have $A\subseteq_{1/8} D$, then  each oriented graph $D'$ on at least~$n_0$ vertices with $\bar\delta^0(D')\ge \ell$ contains $A$. \end{lemma}

\begin{proof}
Given $A$, $\ell$ and $n_0$, let 
 $D'$ be  an oriented graph  on at least $n_0$ vertices with $\bar \delta^0(D')\ge \ell$. Assume that $|\vout(D')|\ge |\vin(D')|$ (the other case is analogous). Then $$|D'\setminus \vin(D')|\ge |D'\setminus \vout(D')|.$$

\begin{figure}[H]
    \centering
    \includegraphics[scale=.8]{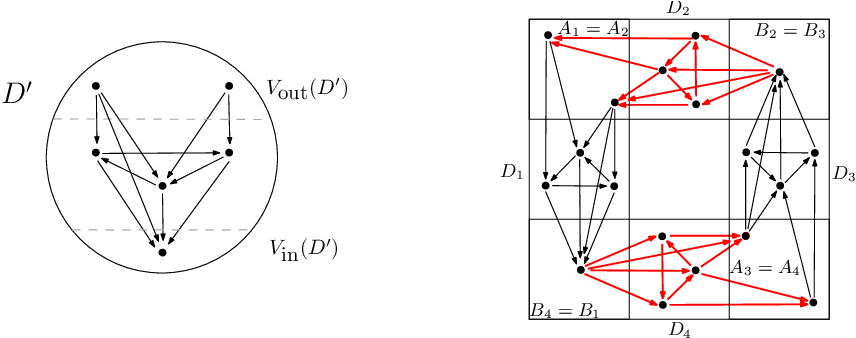}
    \caption{Construction of $D$ from $D'$ in the proof of Lemma~\ref{antilemma}.}
    \label{fig:aux}
\end{figure}

 We will construct an auxiliary oriented graph $D$ to which the hypothesis of the lemma can be applied.
Take four copies $D_1, D_2, D_3, D_4$ of $D'$.  Let $A_i=\vout(D_i)$ and $B_i=\vin(D_i)$ for $i=1,\ldots 4$. For $i=1,3$, identify the vertices of $A_i$ with the vertices of $A_{i+1}$, and the vertices of $B_i$ with the vertices of $B_{i-1}$ (addition modulo $4$). Finally, reverse all directions on edges in copies $D_2$ and  $D_4$. For an illustration, see Figure\ref{fig:aux}.

Note that the obtained digraph $D$ is an oriented graph, and has more than $n_0$ vertices.
Also note that $\delta^0(D)\ge \ell$, since all vertices in $D_i\setminus (A_i\cup B_i)$ maintain their semidegree,  the vertices in $A_i=A_{i+1}$ have  at least $\ell$ out-neighbours in $D_i$ and  at least $\ell$ in-neighbours in $D_{i+1}$, for $i=1,3$, and a similar statement holds for the vertices in the sets $B_i$.

We choose $V^*$ as the set of all vertices in $V(D_1)\setminus B_1$.  Note that $|V^*|\ge \frac {|V(D_1)|}2$ and therefore $|V^*|\ge \frac {|V(D)|}8$. We let $xx'$ be any edge of $A$ (directed from $x$ to $x'$). By the hypothesis of the lemma, we find an embedding of $A$  into $D$, with $x$ embedded into some vertex $v^*\in V^*$. 

We claim  that 
\begin{equation}\label{AinD1}
\text{all vertices of $A$ are embedded into $V(D_1)$.}
\end{equation}

Then, 
as $D_1$ is a copy of $D'$ (with the original directions on the edges), we can conclude that $A$ is a subdigraph of $D'$.

So all that remains is to prove~\eqref{AinD1}. 
Proceeding by contradiction, we assume otherwise, that is, we assume there is a vertex $z\in V(A)$ whose image does not lie in $V(D_1)$. Then the underlying graph of $A$, which is connected by assumption, contains a path from~$x$ to $z$. So $A$, which is antidirected, contains an antipath $xy_1y_2\ldots y_hz$ from $x$ to $z$. Since $xx'$ is an edge of $A$, we know that the edge  $xy_1$ is directed  from~$x$ to $y_1$. 

Recall that $x$ is embedded in $v^*\in V^*=V(D_1)\setminus B_1$. This location of  $v^*$ ensures that  all out-neighbours of  $v^*$ are in $D_1$. More specifically, by the definition of $A_1$, all out-neighbours of  $v^*$ are in 
$D_1\setminus A_1$. In particular, $y_1$ is embedded in a vertex $w_1$ of $D_1\setminus A_1$. This in turn ensures that  all in-neighbours of  $w_1$ are in $D_1\setminus B_1$. In particular, $y_2$ is embedded in a vertex $w_2$ of $D_1\setminus B_1$. Continuing to argue in this manner, we see that all $y_i$ and also $z$ are embedded in $V(D_1)$. 
This proves~\eqref{AinD1}, and thus completes the proof of the lemma.
\end{proof}

\begin{remark}
We remark that Lemma~\ref{antilemma} also holds if both $D$ and $D'$ are directed graphs instead of oriented graphs. However, since  Theorem~\ref{teo:teo1'} only holds for oriented graphs of high minimum semidegree, we need the lemma in the form it is written.
\end{remark}
\begin{remark}
Observe that in Lemma~\ref{antilemma} it is essential to require that a vertex of $A$ can be mapped to a specific set in $D$: If we left out this requirement, then we might find a copy of $A$ in $D_2$ instead of in $D_1$. Thus in $D'$ we would only get a a copy of~$A$ with all directions reversed.
\end{remark}
We are ready to prove Theorem~\ref{ES}.

\begin{proof}[Proof of Theorem~\ref{ES}]
Given $\eta$ and $c$ from Theorem~\ref{ES}, we note that we may assume that $\eta\le 1/8$. Let $n_0$ be given by Theorem \ref{teo:teo1'} for input $\eta$ and $c$. 

Now, let $n\ge n_0$ and $k\ge \eta n$, and  let $D$ be  an oriented graph  on $n$ vertices with $|E(D)|> (1+\eta)(k-1)|V(D)|$. Let $T$ be a balanced antitree with $k$ edges and $\Delta(T)\le (\log n)^c$. We need to show that $T\subseteq D$.

Apply Lemma~\ref{pseu} to $D$ to obtain an oriented graph $D'\subseteq D$ with $\bar \delta^0(D)\ge (1+\eta)\frac k2$. Use Lemma~\ref{antilemma} and Theorem~\ref{teo:teo1'} to see that $T$ can be embedded in $D'$. As $D'\subseteq D$, we see that $T\subseteq D$.
\end{proof}

\section{Conclusion}\label{conclusion}

\paragraph{Oriented trees in digraphs.}
Any digraph $D$ with a minimum semidegree of at least $k$  contains every orientation of every tree with $k$ edges, by a greedy embedding argument. This bound cannot be lowered, not even if we are only looking for oriented paths, as  the disjoint union of complete digraphs of order $k$ has minimum semidegree $k-1$ but no $k$-edge oriented path. Perhaps an additional condition, for instance on the maximum semidegree of~$D$, or requiring $D$ to be weakly connected and sufficiently large, combined with a lower bound on the semidegree of~$D$ could give a variant of Corollary~\ref{coro1} for digraphs.

\paragraph{Oriented trees in oriented graphs.}
In the introduction, we saw the example of a $k$-edge star with all edges directed outwards, which is not contained in, for instance, the $(k-1)$-blow-up of the directed triangle, although this graph has minimum semidegree $k-1$.
We  overcame this difficulty in Theorem~\ref{teo:teo1} by concentrating on balanced antitrees. Another possibility, which was suggested in~\cite{survey}, could be to add an extra condition on the oriented graph $D$, for instance a condition on the maximum semidegree of $D$, with the hope of guaranteeing all $k$-edge antitrees, or even all $k$-edge oriented trees. Such a condition has been successfully used in the undirected graph setting
(see e.g.~\cite{alpha}).

\paragraph{Oriented subdivisions.}
Perhaps Theorem~\ref{teo:teo3} extends to other orientations of subdivisions of the complete graph. For instance, one could ask whether there
is a function $f(h)$ such that every oriented graph of minimum semidegree at least $f(h)$ contains an orientation of a  subdivision of $K_h$, where each path changes directions only a bounded number of times. This would be implied by Conjecture~\ref{conj3}.

\section{Acknowledgment}
The first author would like to thank Matthias Kriesell for pointing out Conjecture~\ref{mader}, and suggesting it might be possible to prove  a result along the lines of Theorem~\ref{teo:teo3}.


\begin{thebibliography}{99}
\makeatletter
\makeatother

\bibitem{achlmt}
{Aboulker, P., Cohen, N., Havet, F., Lochet, W., Moura, P. and Thomass\'e, S.}
{Sub- divisions in Digraphs of Large Out-Degree or Large Dichromatic Number.} 
{\em The Electronic Journal of Combinatorics}, {\bf  26}, P3.19 (2019).

\bibitem{ahlrt}
{Addario-Berry, L., Havet, F., Linhares Sales, C., Reed, B. and Thomassé, S.}, 
{Oriented trees in digraphs}, 
{\em Discrete Mathematics}, {\bf 313} (2013), 967-974. 

\bibitem{tree}
{Ajtai, M., Komlós, J. and Szemerédi, E.}, 
{On a conjecture of {Loebl}}, in
{\em Graph theory, combinatorics, and algorithms, Vol. 1, 2 (Kalamazoo, MI, 1992)}, (1995), 1135–1146. 

\bibitem{direg}
{Alon, N. and Shapira, A.}, 
{Testing Subgraphs in Directed Graphs}, 
{\em Journal of Computer and System Sciences}, {\bf 69} (2004), 354-382. 

\bibitem{alpha}
{Besomi, G., Pavez-Signé, M. and Stein, M.},
{Maximum and minimum degree conditions for embedding trees}, 
{\em SIAM Journal on Discrete Mathematics}, {\bf  34} (2020), 2108-2123.

\bibitem{bsp}
{Besomi, G., Pavez-Signé, M. and Stein, M.}, 
{Degree conditions for embedding trees}, 
{\em SIAM Journal of Discrete Mathematics}, {\bf 33} (2019), 1521–1555. 

\bibitem{burr}
{Burr, S.}, 
{Subtrees of directed graphs and hypergraphs}, 
In {\em Proceedings of the Eleventh South- eastern Conference on Combinatorics, Graph Theory and Combinatorics (Florida Atlantic Univ., Boca Raton, Fla.} (1980), {\bf I,28}, 227-239.

\bibitem{chln}
{Cohen, N. , Havet, F., Lochet, W.  and  Nisse, N.},
{Subdivisions of oriented cycles in digraphs with large chromatic number}, 
{\em Journal of Graph Theory}, {\bf 89(4)} (2018), 439-456.

\bibitem{dirac}
{Dirac, G. A.}, 
{Some theorems on abstract graphs}, 
{\em Proceedings of the London Mathematical Society, 3rd Ser.}, {\bf 2}  (1952), 69-81.

\bibitem{erdosgallai}
{Erdős, P. and Gallai, T.}, 
{On maximal paths and circuits of graphs}, 
{\em Acta Mathematica Academiae Scientiarum Hungarica 10}, {\bf 3} (1959), 337--356. 

\bibitem{gps}
{Girão, A., Popielarz, K., Snyder, R.}
{Subdivisions of digraphs in tournaments,}
{\em Journal of Combinatorial Theory, Series B}
{\bf 146} (2021), 266-285.

\bibitem{gss}
{Gishboliner, L., Steiner, R. and Szabó, T.}
{Dichromatic number and forced subdivisions},
{\em Journal of Combinatorial Theory, Series B}
{\bf 153} (2022), 1-30.

\bibitem{gss2}
{Gishboliner, L., Steiner, R. and Szabó, T.}
{Oriented Cycles in Digraphs of Large Outdegree},
{\em  Combinatorica}, to appear (2022).

\bibitem{graham}
{Graham, R. L. }
{On Subtrees of Directed Graphs with no Path of Length Exceeding One}, 
{\em Canad. Math. Bull.} {\bf 13} (1970), 329-332.

\bibitem{bjackson}
{Jackson, B.}, 
{Long paths and cycles in oriented graphs}, 
{\em Journal of Graph Theory 5}, {\bf 2} (1981), 145-157.

\bibitem{jagger}
{Jagger, C. }
{Extremal digraph results for topological complete tournaments.} {\em European Journal of Combinatorics},
{\bf 19} (1998), 687–694.

\bibitem{km}
{Kathapurkar, A. and Montgomery, R.},
{Spanning trees in dense directed graphs},
{\em J. Combinatorial Theory Series B} {\bf }, (2022), . 


\bibitem{keeko}
{Keevash, P. K\"uhn, D.  and Osthus, O.}
{An exact minimum degree condition for Hamilton cycles in
oriented graphs}, 
{\em J. Lond. Math. Soc.} {\bf 79} (2009), 144-166.

\bibitem{kelly}
{Kelly, L.},
{Arbitrary Orientations of Hamilton Cycles in Oriented Graphs},
{\em Electronic Journal of Combinatorics} {\bf 18(1)} (2011).

\bibitem{kko1}
{Kelly, L., Kühn, D. and Osthus, D.}, 
{A {Dirac} type result on Hamilton cycles in oriented graphs}, 
{\em Combinatorics, Probability and Computing}, {\bf 17} (2008), 689-709.

\bibitem{kko}
{Kelly, L., Kühn, D. and Osthus, D.}, 
{Cycles of given length in oriented graphs}, 
{\em J. Combinatorial Theory Series B} {\bf 100}, (2010), 251-264. 

\bibitem{antipaths}
{Klimošová, T. and Stein, M.}, 
{Antipaths in oriented graphs, in preparation}.

\bibitem{kss}
{Koml\'os, J., S\'ark\"ozy, G. N. and Szemer\'edi, E.},  
{Proof of a packing conjecture of Bollob\'as},
{\em Combinatorics, Probability and Computing}, {\bf 4(3)} (1995), 241-255.

\bibitem{kky}
{Kühn, D., Osthus, D., Young, A.}
{A note on complete subdivisions in digraphs of large outdegree.}
{\em Journal of Graph Theory}, {\bf 57(1)} (2008), 1-6.

\bibitem{mader}
{Mader, W.}, 
{Degree and Local Connectivity in Digraphs}, 
{\em Combinatorica}, {\bf 5} (1985), 161-165.

\bibitem{maderundir}
{Mader, W.}
{Homomorphieeigenschaften und mittlere Kantendichte von Graphen},
{Mathematische Annalen}, {\bf 174} (1967), 265-268.

\bibitem{mn}
{Mycroft, R. and Naia, T.},
{Trees and tree-like structures in dense digraphs},
2020. arXiv preprint: 2012.09201.

\bibitem{scott}
{Scott, A.}
{Subdivisions of Transitive Tournaments.}
{\em European Journal of Combinatorics}
{\bf 21(8)} (2000),  1067-1071.

\bibitem{survey}
{Stein, M.}, 
{Tree containment and degree conditions}, in A. Raigoroskii and M. Rassias, editors,
{\em Discrete Mathematics and Applications}, Springer, (2021), 459–486.

\bibitem{ste}
{Steiner, R.,}
{Subdivisions with congruence constraints in digraphs of large chromatic number}
arXiv:2208.06358 

\bibitem{reg}
{Szemerédi, E.}, 
{Regular Partitions of Graphs}, 
{\em Problémes Combinatoires et Théorie des Graphes Colloques Internationaux CNRS}, {\bf 260} (1978), 399-401.

\bibitem{tho}
{Thomassen, C.},
{Even Cycles in Directed Graphs}, 
{\it European Journal of Combinatorics} {\bf 6}
(1985), 85-89.

\end{thebibliography}
\end{document}